\documentclass{amsart}
\usepackage{amssymb, latexsym}
\usepackage{hyperref}

\let\Re=\undefined
\DeclareMathOperator{\Re}{Re}

\DeclareMathOperator{\Id}{Id}

\newcommand{\D}{{\mathbb D}}
\newcommand{\Z}{{\mathbb Z}}
\newcommand{\C}{{\mathbb C}}
\newcommand{\R}{{\mathbb R}}

\newcommand{\E}{{\mathbb E}}

\newcommand{\calA}{{\mathcal{A}}}
\newcommand{\one}{{\chi}}
\newcommand{\vi}{{\vec\imath}}
\newcommand{\vj}{{\vec\jmath}}
\newcommand{\Grr}{{G_{\!R}^r}}

\newcommand{\qtq}[1]{\quad\text{#1}\quad}

\theoremstyle{plain}
\newtheorem{theorem}{Theorem}
\newtheorem{lemma}[theorem]{Lemma}
\newtheorem{proposition}[theorem]{Proposition}
\newtheorem{corollary}[theorem]{Corollary}
\theoremstyle{definition}
\newtheorem{definition}[theorem]{Definition}
\theoremstyle{remark}
\newtheorem{remark}[theorem]{Remark}

\numberwithin{equation}{section}
\numberwithin{theorem}{section}

\newcounter{smalllist}

\newenvironment{SL}{\begin{list}{{\ \rm(\roman{smalllist})\ }}{%
\setlength{\topsep}{0mm}\setlength{\parsep}{0mm}\setlength{\itemsep}{0mm}%
\setlength{\labelwidth}{1.3em}\setlength{\itemindent}{2em}\setlength{\leftmargin}{0em}\usecounter{smalllist}%
}}{\end{list}}

\begin{document}

\title[Autocorrelation of the characteristic polynomial]{Autocorrelations of the characteristic polynomial of a random matrix\\under microscopic scaling}

\author{Rowan Killip}
\address{University of California, Los Angeles}
\author{Eric Ryckman}
\address{California Institute of Technology}

\maketitle

%
%
%
%

\section{Introduction}

The goal of this paper is to calculate the autocorrelation function for the characteristic polynomial of a
random matrix in the microscopic regime.  As will be explained, results fitting this description have be proved before;
however, here we will cover all values of inverse temperature $\beta\in(0,\infty)$.  The method to be employed
also differs from prior work.

Let us begin by introducing the models to be discussed.  The probability law for the eigenvalues of a matrix chosen
at random from the $n\times n$ unitary group according to Haar measure is given by the Weyl integration formula.
It reads
\begin{align}\label{E:Weyl I F}
\E(f) &=  \tfrac1{n!} \int_{\!-\pi}^\pi\!\! \cdots \! \int_{\!-\pi}^\pi
    f(e^{i\theta_1},\ldots,e^{i\theta_n})
    \bigl|\Delta(e^{i\theta_1},\ldots,e^{i\theta_n})\bigr|^2 \, \tfrac{d\theta_1}{2\pi} \cdots \tfrac{d\theta_n}{2\pi}
\end{align}
for any (symmetric) function $f$ of the $n$ eigenvalues.  Here, $\Delta$ denotes the Vandermonde determinant:
\begin{equation}\label{VDefn}
  \Delta(z_1,\ldots,z_n) = \prod_{1\leq j < k \leq n} \!\! (z_k - z_j).
\end{equation}

The characteristic feature of the eigenvalues of random matrices is their repulsion, expressed in \eqref{E:Weyl I F} by
the second power of $|\Delta|$.  This same power occurs in the setting of random Hermitian matrices; however, for
random real-symmetric or quaternion-self-dual matrices, the power is one or four, respectively.  An analogous
trichotomy occurs in the unitary setting, albeit for certain symmetric spaces, rather than the classical compact Lie
groups $SO(n)$ and $Sp(n)$.  This was discovered by Dyson, \cite{Dyson}, who further advocated studying these three
special cases as a part of the continuum of possible powers of the Vandermonde factor.  This leads to the following
family of probability laws:
\begin{align}  \label{CGbeta}
\E_n^\beta (f) &=  \tfrac{[\Gamma(\frac12\beta + 1)]^n}{\Gamma(\frac12\beta n + 1)} \int_{\!-\pi}^\pi\!\! \cdots \! \int_{\!-\pi}^\pi
    f(e^{i\theta_1},\ldots,e^{i\theta_n})
    \bigl|\Delta(e^{i\theta_1},\ldots,e^{i\theta_n})\bigr|^\beta \, \tfrac{d\theta_1}{2\pi} \cdots \tfrac{d\theta_n}{2\pi}
\end{align}
with $\beta\in[0,\infty)$ and $n$, a non-negative integer.  The parameter $\beta$ is known as the inverse temperature,
consistent with the interpretation of \eqref{CGbeta} as the Gibbs measure for a gas of particles confined to a circle
with logarithmic repulsion (the planar Coulomb law).  For the normalization constant in \eqref{CGbeta}, see
\cite{Good,Wilson}.

Notice that when $\beta=0$, the points $e^{i\theta_j}$ are statistically independent with a uniform distribution on the
circle; this is the infinite-temperature limit.  The $\beta\uparrow \infty$ limit also exists and gives a random
rotation of equi-spaced points on the circle (these configurations give the maximal value for the Vandermonde factor).
In this zero-temperature limit, our gas has frozen into a perfect crystal.

We will study statistics of the `characteristic polynomial'
\begin{align}\label{charpoly}
Z_n(z) := \prod_{j=1}^n \bigl(1 - z^{-1} e^{i\theta_j}\bigr).
\end{align}
When $e^{i\theta_j}$ are interpreted as the eigenvalues of a matrix, this is indeed the characteristic polynomial,
except that a factor $z^n$ is missing.  This (re)normalization of the characteristic polynomial is rather popular in
random matrix theory, particularly in papers drawing analogies with the Riemann zeta function.

The nature of our goals in this paper is most easily seen by looking at the simplest non-trivial example:

\begin{theorem}\label{T:2 points}  Writing $\gamma=\tfrac2\beta$, we have
\begin{equation}\label{Two point convariance}
  \lim_{n\to\infty} n^{-\gamma} \, \E^\beta_n\Bigl\{ Z_n(e^{ix/n}) \overline{Z_n(e^{-ix/n})} \Bigr\} =
    \pi^{\frac12} (2x)^{\frac12 - \gamma} e^{-ix} J_{\gamma-\frac12}(x)
\end{equation}
for any $x\in\R$.  Here $J_\nu$ denotes the Bessel function of the first kind and order $\nu$.
\end{theorem}

This result is a special case of Corollary~\ref{C:Answer 2}. The most general result in this paper is Theorem~\ref{T:main} which evaluates
autocorrelations
\begin{equation}\label{Multi point convariance}
  \lim_{n\to\infty} n^{-2rq/\beta} \, \E^\beta_n\biggl\{ \prod_{j=1}^{q} Z_n(e^{iw_j/n})
        \prod_{k=1}^r \overline{Z_n(e^{iy_k/n})} \biggr\},
\end{equation}
for general tuples of complex numbers $(w_1,\ldots,w_q)$ and $(y_1,\ldots,y_r)$, in terms of the solution to a certain
system of linear ODEs.  While unable to give the general solution of the relevant systems of ODEs, our results still
reveal something.  In particular, we see that \eqref{Multi point convariance} is an analytic function of all
parameters, not only $w_j$ and $y_k$, but also of $\beta$.  This indicates that no phase transition takes place, at
least at the level of the characteristic polynomial.

Notice that in \eqref{Two point convariance} and \eqref{Multi point convariance} we are sending $n$, the number of
particles (or eigenvalues), to infinity, while rescaling the locations at which we evaluate the characteristic polynomial by
$1/n$.  This is termed the microscopic thermodynamic limit; it reveals behaviour at the scale of the typical
inter-particle distance amid a sea of particles.  It is in this scaling that random matrix behaviour is believed to be
universal; see, for example, \cite{Berry,Bohigas,KatzSarnak,Mehta,Montgomery}.

As noted earlier, the case $\beta=2$ of our model corresponds to the eigenvalues of a random element of the unitary
group, while two further special values, namely, $\beta=1$ and $\beta=4$, arise as eigenvalue distributions
for certain symmetric spaces of matrices.  These three models are completely integrable in some
sense, in particular, the correlation functions for the point processes have simple determinantal/Pfaffian expressions.
In these three cases, the study of moments of the characteristic polynomial is fully developed: not only have the
autocorrelations \eqref{Multi point convariance} been evaluated, but their values for finite $n$ and for rational
functions of $Z_n$ are also known.  A sampling of the work in this direction can be found in the papers
\cite{AkemannV,AndreevSimons,BasorForrester,BorodinStrahov,BrezinHikami,BaikDeiftStrahov,ConreyFKRSnaith,MehtaN,StrahovFyod}.  We draw
particular attention to \cite{BorodinStrahov} which completed the program in the Gaussian case and provides further references.

Very little appears to be known about the case of general $\beta$; certainly the asymptotics of rational functions of $Z_n$ are
unknown, for these would determine the correlation functions of the point process, which is an outstanding open problem.
Two papers of particular note are \cite{Aomoto,Kaneko}, which consider characteristic polynomials for the Jacobi ensemble
at general temperature from the perspective of the Selberg integral.  The paper of Aomoto computes the expected value of the
characteristic polynomial for general size $n$ and point $x\in\C$; the answer is essentially a Jacobi polynomial.  (For a proof
of this result via the approach of this paper, see \cite[Proposition~6.1]{KN1}.)

The paper \cite{Kaneko} of Kaneko discusses autocorrelations of the characteristic polynomial at finite $n$.  It is proved that the
autocorrelations obey a system of second-order PDEs with the locations at which the characteristic polynomial is evaluated as independent
variables.  It is also shown that the autocorrelations can be expressed as a hypergeometric function of matrix argument (in the sense
of \cite{Herz}); this is defined as an infinite series (with summation over partitions) of Jack polynomials.

In this paper we apply simple analytical methods combined with a change of variables inspired by the theory of orthogonal polynomials.
While the approach described here can also be applied in the Jacobi setting, by using \cite[Theorem~1.5]{KN1} in place of
Theorem~\ref{T:KN_C} below, the development would be quite messy.  We have chosen to confine our attention to the circular case
because it most clearly shows both the virtues and the limitations of our method.

\section*{Acknowledgements}
The first author was supported, in part, by NSF grant DMS-0701085.

\section{Preliminaries}

We first give a quick review of the connection between the eigenvalue problem and orthogonal polynomials;
see \cite{OPUC1} for further information.

Let $U$ be an $n\times n$ unitary matrix $U$ for which
\begin{equation}\label{powers}
e_1,\ Ue_1,\ U^2 e_1,\ U^3e_1,\ \ldots, U^{n-1} e_1
\end{equation}
are linearly independent; here $e_1=[1,0,\ldots,0]^T$.  This condition holds generically, though it does require that all eigenvalues
be simple.  Applying the Gram-Schmidt procedure to the vectors in \eqref{powers}, we find a sequence of \emph{monic} polynomials $\Phi_k(z)$
so that $\{ \Phi_k(U)e_1 \}$ is an othogonal set.  Note that $\Phi_0(z)=1$.  If we define `reversed' polynomials
\begin{equation}\label{Phi*}
\Phi_{k}^*(z) = z^k \overline{\Phi_{k}(\bar z^{-1})},
\end{equation}
then both $[\Phi_{k+1}(U) - U\Phi_k(U)]e_1$ and $\Phi^*_k(U)e_1$ are perpendicular to $\{Ue_1,\ldots,U^{k}e_1\}$ and so must be co-linear.  This leads
us to the recurrence relations
\begin{equation}
\begin{aligned}\label{SzegoRec}
\Phi_{k+1}(z) &= z\Phi_{k}(z) - \bar\alpha_k \Phi_{k}^*(z) \\
\Phi_{k+1}^*(z) &= \Phi_{k}^*(z) - \alpha_k z\Phi_{k}(z)
\end{aligned}\quad\text{with}\quad \bar\alpha_k = \frac{\langle\Phi^*_k(U)e_1,U\Phi_k(U)e_1\rangle}{\|\Phi^*_k(U)e_1\|^2}\in \D
\end{equation}
where $\D$ denotes the open unit disk in the complex plane.  For $0\leq k < n-1$, as above, $U\Phi_k(U)e_1$ cannot be a multiple of $\Phi^*_k(U)e_1$
since it would contradict linear independence in \eqref{powers}; this shows $|\alpha_k|<1$.  Running the same argument with $k=n-1$ reveals
\begin{equation}\label{Z from Phi}
z^n Z_n(z) = z \Phi_{n-1}(z) - e^{-i\eta} \Phi^*_{n-1}(z) \quad\text{for some $\eta\in[0,2\pi)$;}
\end{equation}
indeed, $\det(U) = (-1)^{n-1} e^{-i\eta}$.  Note that the parameters $\{\alpha_0,\ldots,\alpha_{n-1},\eta\}$ do not determine the
matrix $U$ uniquely, but merely up to a change of basis fixing $e_1$.  There are several systems of canonical representatives
for these equivalence classes. One such system, known as CMV matrices, \cite{CMV,Watkins}, is to be noted for its sparsity.

Following the prevailing parlance among those working with orthogonal polynomials, we will refer to
$(\alpha_0,\ldots,\alpha_{n-1},\eta)$ as the \emph{Verblunsky coefficients}.

If we choose $U$ at random according to Haar measure, then the Verblunsky coefficients are also random. Their joint law
was computed in \cite{KN1} by mimicking an argument of Trotter, \cite{Trotter}, in a related self-adjoint setting. They
are statistically independent with $\alpha_k \sim \Theta_{2(n-k)-1} $ and $\eta$ uniform on $[0,2\pi)$.  The
$\Theta_\nu$ probability distribution is defined as follows:

\begin{definition}  A complex random variable, $\alpha$, with values in the unit disk,
$\D$, is \emph{$\Theta_\nu$-distributed} (for $\nu>1$) if
\begin{equation}\label{E:ThetaDefn}
\E\{f(\alpha)\} = \tfrac{\nu-1}{2\pi} \int\!\!\!\int_\D f(z)
(1-|z|^2)^{(\nu-3)/2} \,d^2z.
\end{equation}
The $\nu=1$ limit corresponds to $\alpha$ uniformly distributed on the unit circle, $\partial\D$.  The $\nu\to\infty$
limit, which is relevant to the case $\beta=\infty$ (i.e., zero temperature), corresponds to $\alpha\equiv 0$.
\end{definition}

\begin{remark}\label{R:rot}
The distribution is rotationally invariant, that is, $\alpha$ and $e^{i\theta}\alpha$ follow the same law.  This
observation results in significant simplifications in what follows.
\end{remark}

\begin{remark}\label{R:ThetaMoments}
The moments of $\alpha\sim\Theta_\nu$ are given by
\begin{equation}\label{E:ThetaMoments}
\E\bigl\{ \alpha^p \bar \alpha^q \bigr\} = \delta_{pq} \frac{2^p \, p!}{(\nu+1)(\nu+3)\cdots(\nu+2p-1)},
\end{equation}
as is easily seen by switching to polar coordinates and recognizing Euler's Beta integral.
\end{remark}

Directly inspired by work of Dumitriu and Edelman in the self-adjoint case, \cite{DumE}, Killip and Nenciu proved

\begin{theorem}[{\cite[Theorem~1.2]{KN1}}]\label{T:KN_C}
Fix $\beta\in[0,\infty]$, let $\alpha_k\sim\Theta_{\beta(k+1)+1}$ be independent random variables, and let $e^{i\eta}$
be independent and uniformly distributed on~$\partial\D$.  Then the zeros of the function
\begin{align}\label{Z from Phi 2}
Z_n(z) &= z^{1-n} \Phi_{n-1}(z) - e^{-i\eta} z^{-n} \Phi^*_{n-1}(z),
\end{align}
defined by solving the recurrence \eqref{SzegoRec} with initial data $\Phi_0(z)=\Phi^*_0(z)=1$,
are distributed on the unit circle according to \eqref{CGbeta}.
\end{theorem}

Note that the laws for the parameters $\alpha_k$ is reversed relative to the case of the unitary group discussed
earlier (for which $\beta=2$).  Proposition~B.2 from \cite{KN1} shows that this does not affect the distribution of the
`eigenvalues'; however it does simplify many formulae below.

From \eqref{Z from Phi 2} and \eqref{Phi*}, we see that
\begin{equation}\label{Z bar}
\overline{ Z_n(e^{ix}) } = - e^{i\eta+in\bar x} Z(e^{i\bar x})
\end{equation}
for any $x\in\C$.  In this way, our basic object of investigation can be re-written as follows:
\begin{equation}\label{E:Z^R}
\begin{aligned}
\E^\beta_n\biggl\{ \prod_{j=1}^q Z_n & (e^{iw_j/n}) \prod_{k=1}^r \overline{Z_n(e^{iy_k/n})} \biggr\} \\
   &= (-1)^r \biggl(\prod_{k=1}^r e^{i\bar y_k}\biggr) \E^\beta_n\biggl\{ e^{ir\eta} \prod_{p=1}^R Z_n\bigl(e^{ix_p/n}\bigr) \biggr\}
\end{aligned}
\end{equation}
where $R=q+r$ and the vector $\vec{x}$ is defined via
\begin{equation}\label{E:concat}
\vec x = (x_1,\ldots,x_R) = (w_1,\ldots,w_{R-r},\bar y_1,\ldots, \bar y_r).
\end{equation}
This reduces our task to something of a much more symmetric form.  In the next section we upgrade the recurrence
\eqref{SzegoRec} to a recurrence for $R$-fold products of orthogonal polynomials evaluated at $R$ different points, as
is required to attack \eqref{E:Z^R}.  This reduces the problem to the analysis of a matrix product.  Determining the
dominant eigenvalue of each factor requires the following:

\begin{lemma}\label{L:blah}
If $\alpha$ is $\Theta_\nu$-distributed and $\Re \lambda+\mu > -1$, then
\begin{equation}\label{blah1}
\E\Big\{ (1-\alpha)^{\lambda}(1-\bar\alpha)^\mu \Big\}
= \frac{ \Gamma(\lambda+\mu+\tfrac{\nu+1}2)\,\Gamma(\tfrac{\nu+1}2) }%
    { \Gamma(\lambda+\tfrac{\nu+1}2)\Gamma(\mu+\tfrac{\nu+1}2) }.
\end{equation}
Sending $\nu\downarrow 1$ we obtain a special case of Euler's beta integral:
\begin{equation}\label{blah2}
\int_0^{2\pi} (1-e^{i\eta})^{\lambda}(1-e^{-i\eta})^{\mu} \frac{d\eta}{2\pi}
=  \frac{ \Gamma(\lambda+\mu+1) }{ \Gamma(\lambda+1)\Gamma(\mu+1) }.
\end{equation}
\end{lemma}

\begin{proof}
This corresponds to Lemma~2.3 in \cite{BHNY}.  The key step in their proof is the following hypergeometric sum, which
is due to Gauss (cf. Theorem 2.2.2 in \cite{AndrewsAskeyRoy}):
$$
\sum_{p=0}^{\infty} \frac{(-\lambda)_p(-\mu)_p}{p!(\frac{\nu+1}{2})_p}
    = \text{RHS\eqref{blah1}}
\quad\text{where}\quad
(z)_p=z(z+1)(z+2)\cdots(z+p-1).
$$
Indeed, the reader should have little difficulty in reconstructing the proof from this identity and
\eqref{E:ThetaMoments}.
\end{proof}

%
%
%
%

\section{The basic recursion}

Given (random) Verblunsky coefficients $(\alpha_0,\ldots,\alpha_{n-1},\eta)$ and $x\in\C$ we define
$$
v_k(x) := \begin{bmatrix} e^{ikx/2} \Phi_k^*(e^{ix}) \\ e^{-ikx/2} \Phi_k(e^{ix}) \end{bmatrix}
$$
where $\Phi_k(z)$ and $\Phi_k^*(z)$ are solutions of the recurrence \eqref{SzegoRec} with initial data
$\Phi_k(z)\equiv\Phi_k^*(z)\equiv1$.  Correspondingly,
\begin{equation}\label{E:v recurr1}
v_{k+1}(x) = A_k(x) v_{k}(x) \qtq{with}
    v_0(x) = \begin{bmatrix} 1 \\1 \end{bmatrix}
\end{equation}
and
\begin{equation}\label{E:v recurr2}
   A_k(x):= \begin{bmatrix} e^{-ix/2}  & -\alpha_k e^{ix/2} \\ -\bar\alpha_k e^{-ix/2} & e^{ix/2} \end{bmatrix}
    = \begin{bmatrix} 1  & -\alpha_k \\ -\bar\alpha_k & 1 \end{bmatrix}
       \begin{bmatrix} e^{-ix/2}  & 0 \\ 0 & e^{ix/2} \end{bmatrix}
\end{equation}
For the model that we are considering (cf. Theorem~\ref{T:KN_C}), $v_k$ and $\alpha_k$ are statistically independent.
Thus, writing $\vec x = (x_1,\ldots,x_R)$ as before,
$$
\E\{v_{k+1}(x_1)\otimes\cdots\otimes v_{k+1}(x_R)\} = \calA_k(\vec x)
\E\{v_{k}(x_1)\otimes\cdots\otimes v_{k}(x_R)\}
$$
with
\begin{equation}\label{E:Acal defn}
\calA_k(\vec x) = \E\{ A_k(x_1) \otimes\cdots\otimes A_k(x_R) \}.
\end{equation}

In view of \eqref{Z from Phi 2} we have
\begin{equation}\label{E:Z from v}
Z_n(e^{ix/n}) = e^{-i(n-1)x/2n} v_{n-1,2}(x/n) - e^{-i\eta} e^{-i(n+1)x/2n} v_{n-1,1}(x/n).
\end{equation}
Thus, from \eqref{E:Z^R} we see that our goal is to determine the $n\to\infty$ asymptotics of
$$
\calA_{n-2}(\vec{x}/n) \calA_{n-3}(\vec{x}/n) \cdots \calA_{1}(\vec{x}/n) \calA_{0}(\vec{x}/n) .
$$
More precisely, we need the asymptotics of this matrix applied to the vector of ones.

Entries in the large tensor products above are naturally indexed by elements of $G_R:=\{1,2\}^R$, which we will
typically denote by $\vi=(i_1,\ldots,i_R)$ or similarly $\vj$.  We will think of the underlying vector space as
$\ell^2(G_R)$.

We have called the index set $G_R$ because it is instructive to regard it as the vertex set of a graph.  Specifically,
we define adjacency through the adjacency matrix
\begin{equation}\label{E:Delta defn}
\Delta_{\vi\,\vj} = \begin{cases} 1 & \text{if}\ \ \#\{q:i_q\neq j_q\}=2 \ \ \text{and}\ \ \sum_q (i_q-j_q)=0 \\
        0 & \text{otherwise.} \end{cases}
\end{equation}
Notice that two vertices are joined if they differ by exactly one down-flip ($1\mapsto2$) and one up-flip
($2\mapsto1$).  The up/down nomenclature corresponds to the way one indexes column vectors of length two.

The graph $G_R$ has $R+1$ connected components
$$
\Grr = \{ \vi : \textstyle\sum i_q = 2R-r \}
$$
where $r$ runs over $\{0,1,\ldots,R\}$.  Less cryptically, $\Grr$ contains those $\vi$ that consist of $r$ copies of
$1$ and $R-r$ copies of $2$.  Of course, $\ell^2(\Grr)$ is a reducing subspace for $\Delta$; we will also write
$\Delta$ for its restriction to this space.

Given $\vec{x}=(x_1,\ldots,x_R)\in\C^R$ we define a (complex) potential $V=V(\vec{x})$
\begin{equation}\label{E:V defn}
V_{\vi\,\vj} = \begin{cases} \tfrac{\sqrt{-1}}2 \sum_q (-1)^{i_q}x_q & \text{if} \ \ \vi=\vj \\ 0 & \text{otherwise.} \end{cases}
\end{equation}
We use the name `potential' to maintain the quantum-mechanical analogy we began by writing $\Delta$ for the adjacency
matrix (and regarding it as a Laplacian).  As for $\Delta$, we maintain the name $V$ for the restriction of this
operator to the invariant subspaces $\ell^2(\Grr)$.

\begin{proposition}\label{P:all about calA}
For each $\vec{x}=(x_1,\ldots,x_R)\in\C^R$ and integers $0\leq k<n$, the space $\ell^2(\Grr)$ is invariant for
$\calA_k(\vec{x}/n)$.  Moreover, we have the following:
\begin{SL}
\item The matrix $\calA_k(\vec{0})$ is real-symmetric, while for general $\vec{x}\in\C^R$,
\begin{equation}\label{E:calA:x neq 0}
\calA_k(\vec{x}/n) = \calA_k(\vec{0}) \, e^{V/n}
\end{equation}
where $V=V(\vec{x})$ is the matrix defined in \eqref{E:V defn}.
\item The characteristic function $\one_R^r\in\ell^2(G_R)$ of $\Grr$ is an eigenvector for $\calA_k(\vec{0}):$
\begin{equation}\label{E:calA:principal e-val}
\calA_k\bigl(\vec{0}\bigr) \one_R^r = \frac{ \Gamma(R+1+\tfrac{\beta}2(k+1))\,\Gamma(1+\tfrac{\beta}2(k+1)) }%
    { \Gamma(R-r+1+\tfrac{\beta}2(k+1))\Gamma(r+1+\tfrac{\beta}2(k+1)) } \one_R^r
\end{equation}
\item For $0\leq k <n$,
\begin{equation}\label{E:calA:asymptotic}
\calA_k\bigl(\tfrac{\vec{x}}{n}\bigr) = \Id + \tfrac{2}{\beta(k+1)+1}\Delta + \tfrac{1}{n}V
    + O\bigl((k+1)^{-2}).
\end{equation}
\item\label{I:Delta} $r(R-r)-\Delta$ is positive semi-definite as an operator on $\ell^2(\Grr)$.  Indeed,
$$
    \Delta\one_R^r=r(R-r)\one_R^r
$$
while no other eigenvalues have greater modulus.
\item For $0\leq k \leq m <n$
\begin{equation}\label{E:calA:prod norm}
\bigl\| \calA_m\bigl(\tfrac{\vec{x}}{n}\bigr) \cdots \calA_k\bigl(\tfrac{\vec{x}}{n}\bigr) \|_{\ell^2(\Grr)\to\ell^2(\Grr)}
    = O\bigl( (\tfrac{m+1}{k+1})^{2r(R-r)/\beta} \bigr).
\end{equation}
\end{SL}
In \eqref{E:calA:asymptotic} and \eqref{E:calA:prod norm}, the implicit constants depend on $\vec{x}$, but not $n$.
\end{proposition}

\begin{proof}
From the rotation invariance of the law of $\alpha_k$ we see that only those entries in $A_k\otimes\cdots\otimes A_k$ that
correspond to equal numbers of up and down flips can have non-zero expectation.  Indeed, looking back to \eqref{E:v recurr2},
we see that each up-flip brings with it a factor of $\alpha_k$ and each down-flip, a factor of
$\bar\alpha_k$.  Thus $\ell^2(\Grr)$ is invariant subspace for $\calA_k(\vec{x}/n)$.  Taking this a step further in the
case $\vec{x}=\vec{0}$ we see that
\begin{equation}\label{E:calA:matrix elements}
\bigl( \calA_k(\vec{0}) \bigr)_{\vi\,\vj} = \delta_{\vi\,\vj} + \E\Bigl\{|\alpha_k|^{\sum |i_q-j_q|}\Bigr\}
    \qtq{for all} \vi,\vj\in \Grr.
\end{equation}
Not only is this real-symmetric, but all entries in this sub-matrix are positive.  Thus the full weight of the
Perron--Frobenius Theorem applies.  In particular, there is a unique eigenvalue of greatest modulus, it is positive,
simple, and the corresponding eigenvector has positive entries. Since all row sums are the same, the principal
eigenvector is the vector of ones, $\one_R^r$. To compute the principal eigenvalue, we evaluate the row sum: For all
$\vi\in \Grr$
\begin{align*}
\sum_{\vj\in \Grr} \E\Bigl\{|\alpha_k|^{\sum |i_q-j_q|}\Bigr\}
    &= \sum_{p=0}^{\min(r,R-r)} \binom{r}{p}\binom{R-r}{p} \E\bigl\{ |\alpha|^{2p} \bigr\} \\
&= \E \bigl\{ (1-\alpha)^r (1-\bar\alpha)^{R-r} \bigr\} \\
&= \frac{ \Gamma(R+1+\tfrac{\beta}2(k+1))\,\Gamma(1+\tfrac{\beta}2(k+1)) }%
    { \Gamma(R-r+1+\tfrac{\beta}2(k+1))\Gamma(r+1+\tfrac{\beta}2(k+1)) }.
\end{align*}
Note the use the rotation invariance of the law of $\alpha_k$ to obtain the second equality and the use of
Lemma~\ref{L:blah} for the third.

Equation \eqref{E:calA:x neq 0} follows immediately from the definition of $V$ and the right-hand identity in
\eqref{E:v recurr2}.

To obtain \eqref{E:calA:asymptotic} when $\vec{x}=\vec{0}$, which implies $V\equiv 0$, we simply combine
\eqref{E:ThetaMoments} and \eqref{E:calA:matrix elements}.  The case of general (i.e., non-zero) $\vec{x}$ follows from
this and \eqref{E:calA:x neq 0}.

To verify part~\eqref{I:Delta} of the proposition, we apply the Perron--Frobenius Theorem to $\Delta$ acting on
$\ell^2(\Grr)$. The principle eigenvalue is $r(R-r)$ because this is the number of neighbours of each vertex in $\Grr$.

To finish the proof of the proposition we need to verify \eqref{E:calA:prod norm}.  Using \eqref{E:calA:x neq 0}, then
the $\vec{x}=\vec{0}$ case of \eqref{E:calA:asymptotic} and part~\eqref{I:Delta} of Proposition~\ref{P:all about calA} yields
\begin{align*}
\bigl\| \calA_m\bigl(\tfrac{\vec{x}}{n}\bigr) \cdots \calA_k\bigl(\tfrac{\vec{x}}{n}\bigr) \|_{\ell^2(\Grr)\to\ell^2(\Grr)}
    &\leq e^{(m-k)\|V\|/n} \bigl\| \calA_m\bigl(\vec{0}\bigr) \cdots \calA_k\bigl(\vec{0}\bigr) \|_{\ell^2(\Grr)\to\ell^2(\Grr)} \\
&= O\biggl( \exp\biggl\{ \sum_{p=k}^{m} \frac{2r(R-r)}{\beta(p+1)+1} \biggr\} \biggr).
\end{align*}
Bounding the sum by an integral easily yields the result.
\end{proof}

In view of \eqref{E:calA:principal e-val}, we need to evaluate the asymptotics of a product of Gamma functions even
just to solve our problem in the case $\vec{x}=0$.  This is what we do next.

\begin{lemma}\label{L:prod asymp}  For $\beta\in(0,\infty)$,
\begin{equation}\label{E:prod asymp}
\!\binom{R}{r} \prod_{k=0}^{n-2} \frac{ \Gamma(R+1+\tfrac{\beta}2(k+1))\,\Gamma(1+\tfrac{\beta}2(k+1)) }%
    { \Gamma(R-r+1+\tfrac{\beta}2(k+1))\Gamma(r+1+\tfrac{\beta}2(k+1)) }
        = (C+O(\tfrac1n)) n^{\frac{2r(R-r)}{\beta}} \!\!\!
\end{equation}
as $n\to\infty$ where $C=C(r,R,\beta)$ is given by
\begin{equation}\label{E:Cdefn}
  C:= \biggl(\prod_{p=1}^{r} \frac{\Gamma(\tfrac2{\beta}p)}{\Gamma(\tfrac2{\beta}(R-r+p))} \biggr).
\end{equation}
\end{lemma}

\begin{proof}
Noting that the binomial coefficient corresponds to a $k=-1$ term in the product and setting $l=k+1$, yields
\begin{align*}
\text{LHS\eqref{E:prod asymp}}
&= \prod_{l=0}^{n-1} \biggl( \frac{ \Gamma(R+1+\tfrac{\beta}2 l)}{ \Gamma(R-r+1+\tfrac{\beta}2 l)}
    {}\div{} \frac{ \Gamma(r+1+\tfrac{\beta}2 l) }{ \Gamma(1+\tfrac{\beta}2 l) } \biggr)\\
&= \prod_{l=0}^{n-1} \prod_{p=1}^{r} \frac{R-r+p+\tfrac{\beta}2 l}{p+\tfrac{\beta}2 l}\\
&= \prod_{p=1}^{r} \prod_{l=0}^{n-1} \frac{l + \tfrac2{\beta}(R-r+p)}{l + \tfrac2{\beta}p}\\
&= \prod_{p=1}^{r} \frac{\Gamma\bigl(n+\tfrac2{\beta}(R-r+p)\bigr)}{\Gamma\bigl(\tfrac2{\beta}(R-r+p)\bigr)} \frac{\Gamma(\tfrac2{\beta}p)}{\Gamma(n+\tfrac2{\beta}p)}
\end{align*}
The result now follows from
$$
    n^{b-a} \frac{\Gamma(n+b)}{\Gamma(n+a)} = 1 + O(\tfrac1n),
$$
which is a consequence of Stirling's formula (cf. \cite[Theorem~1.4.2]{AndrewsAskeyRoy}).
\end{proof}

\begin{lemma}\label{L:A prod asymp}  Fix $\beta\in(0,\infty)$, integers $0\leq r \leq R$ and $\vec{x}\in\C^R$.
Let $\one_R^r\in\ell^2(G_R)$ be the characteristic function of $\Grr$, then for any integers $0<m<n$,
\begin{align*}
m^{-2r(R-r)/\beta} \tbinom{R}{r} \calA_{m-1}(\vec{x}/n) \cdots \calA_{1}(\vec{x}/n) \calA_{0}(\vec{x}/n) \one_R^r
&= C  \, \one_R^r + O\bigl( \tfrac{m}n + \tfrac{1}{m} \bigr)
\end{align*}
where $C$ is the constant given by \eqref{E:Cdefn}. The constant implicit in the $O$ notation depends on $\vec{x}$ but
may be chosen independent of $m$ and $n$.
\end{lemma}

\begin{proof}
Using \eqref{E:calA:x neq 0} and \eqref{E:calA:prod norm},
\begin{align*}
& \bigl\| \calA_{m-1}(\vec{x}/n) \cdots \calA_{0}(\vec{x}/n) - \calA_{m-1}(\vec{0}) \cdots \calA_{0}(\vec{0})\bigr\| \\
\leq {}& m  \|e^{V/n} -1\| \cdot O\bigl( m^{2r(R-r)/\beta} \bigr) \\
={} & O \bigl(\tfrac{m}n m^{2r(R-r)/\beta} \bigr).
\end{align*}
Where all norms are in the sense of operators on $\ell^2(\Grr)$.  This reduces the problem to the case $\vec{x}=0$, for
which we apply \eqref{E:calA:principal e-val} and Lemma~\ref{L:prod asymp}:
$$
\bigl\| \tbinom{R}{r} \calA_{m-1}(\vec{0}) \cdots \calA_{0}(\vec{0}) \one_R^r - C \one_R^r m^{\frac{2r(R-r)}{\beta}} \bigr\|
    = O\bigl( \tfrac1m \;\!m^{2r(R-r)/\beta} \bigr),
$$
so completing the proof of the lemma.
\end{proof}

\begin{lemma}\label{L:ODE}
Fix $\beta\in(0,\infty)$ and integers $0\leq r \leq R$. For each $\vec{x}\in\C^R$, there is a unique $C^1$ function
$\Psi(\cdot,\vec{x}):(0,\infty)\to \ell^2(\Grr)$ such that
\begin{equation}\label{E:ODE}
\frac{d\ }{dt} \Psi(t,\vec{x}) = \bigl( \tfrac{2}{\beta t} \Delta + V \bigr) \Psi(t,\vec{x})
    \qtq{and}  t^{-2r(R-r)/\beta} \Psi(t) =  \one_R^r + O(t)
\end{equation}
as $t\downarrow 0$. Here $\one_R^r$ is the vector of ones in
$\ell^2(\Grr)$. Moreover, $\Psi$ is an entire function of $\vec{x}$. It is entire in $t$ except for a possible branch
point at $t=0$ and admits a holomorphic continuation in $\beta$ outside the (real) interval $(-2r(R-r),0)$.
\end{lemma}

\begin{proof}
Existence and analyticity are not difficult. Indeed, for integers $k\geq0$ let us define vectors $\Psi_k$ by
$$
\Psi_{k+1} = \bigl(k+1+\tfrac2\beta[r(R-r)-\Delta]\bigr)^{-1} V \Psi_k \qtq{with} \Psi_0 = \one_R^r.
$$
Then the infinite series
$$
    t^{2r(R-r)/\beta} \sum_{k=0}^\infty \Psi_k t^k
$$
converges and gives a solution to \eqref{E:ODE} with the correct behaviour for small $t$.  Indeed, this follows from two
facts proved in Proposition~\ref{P:all about calA}: $r(R-r)-\Delta$ is positive semi-definite and $\Delta
\one_R^r=r(R-r)\one_R^r$.

This leaves us to address the question of uniqueness.  Let $\Phi(t)=t^{-2r(R-r)/\beta}\Psi(t)$, then by variation
of parameters,
\begin{equation}\label{E:prePicard}
\Phi(t) = \bigl(\tfrac{t_0}{t}\bigr)^{2[r(R-r)-\Delta]/\beta} \Phi(t_0) + \int_{t_0}^t \bigl(\tfrac{s}{t}\bigr)^{2[r(R-r)-\Delta]/\beta} V \Phi(s)\,ds
\end{equation}
for any $0<t_0<t$.  Thus sending $t_0\downarrow 0$ we deduce that
$$
\Phi(t) = \one_R^r + \int_0^t \bigl(\tfrac{s}{t}\bigr)^{2[r(R-r)-\Delta]/\beta} V \Phi(s)\,ds
$$
and so uniqueness (as well as another proof of existence) follow by the usual Picard iteration argument.
\end{proof}

\begin{remark}
Our system of ODEs \eqref{E:ODE} has a scaling symmetry and hence so does the solution.  More precisely,
$$
\Psi(\lambda t,\lambda^{-1} \vec{x}) = \lambda^{2r(R-r)/\beta} \Psi(t, \vec{x}) \qtq{for any} \lambda>0.
$$
Later we will see that only the value of $\Psi(t=1,\vec{x})$ is needed to determine the autocorrelation with parameters $\vec x$;
nevertheless this symmetry shows the additional information in $\Psi(t=1,\vec{x})$ is not wholly redundant.
\end{remark}

\begin{lemma}  For $0<m\leq n-2$,
\begin{equation}\label{E:prod as ODE}
\calA_{n-2}(\vec{x}/n) \cdots \calA_{m}(\vec{x}/n) \Psi(\tfrac mn,\vec{x}) = \Psi(1,\vec{x})
    + O\bigl( \tfrac1m \bigl(\tfrac{n}{m}\bigr)^{2r(R-r)/\beta}\bigr).
\end{equation}
\end{lemma}

\begin{proof}
From \eqref{E:calA:asymptotic} and \eqref{E:prePicard},
$$
\Psi(\tfrac{k+1}n,\vec{x}) = \calA_{k}(\vec{x}/n) \Psi(\tfrac{k}n,\vec{x}) + O(k^{-2}).
$$
The result now follows easily from this, \eqref{E:calA:prod norm}, and
$$
\sum_{k=m}^{n-2} k^{-2} \bigl(\tfrac{n}{k+1}\bigr)^{2r(R-r)/\beta} \leq 2 m^{-1} \bigl(\tfrac{n}{m}\bigr)^{2r(R-r)/\beta} ,
$$
which requires only elementary manipulations.
\end{proof}

Putting everything together yields

\begin{proposition}\label{P:v Answer}
For $\vi\in \Grr$,
$$
\lim_{n\to\infty} n^{-2r(R-r)/\beta} \tbinom Rr \E\{v_{n-1}(x_1/n)\otimes\cdots\otimes v_{n-1}(x_R/n)\}_\vi = C \Psi(t=1,\vec{x})_\vi \,.
$$
Here $C$ is as in \eqref{E:Cdefn}.
\end{proposition}

\begin{proof}
From the small $t$ behaviour of $\Psi(t)$ we have
$$
m^{2r(R-r)/\beta} \one_R^r = n^{2r(R-r)/\beta} \Psi(\tfrac mn,\vec{x}) + O\bigl(\tfrac mn m^{2r(R-r)/\beta}\bigr).
$$
Combining this with Lemma~\ref{L:A prod asymp} yields
\begin{align*}
\tbinom{R}{r} \calA_{m-1}(\vec{x}/n) \cdots \calA_{1}(\vec{x}/n) & \calA_{0}(\vec{x}/n) \one_R^r\\
    &= C \Psi(\tfrac mn,\vec{x}) n^{2r(R-r)/\beta} + O\bigl([\tfrac mn + \tfrac1m] m^{2r(R-r)/\beta}\bigr).
\end{align*}
Next, we employ \eqref{E:prod as ODE} and \eqref{E:calA:prod norm} to obtain
\begin{align*}
n^{- 2r(R-r)/\beta} \tbinom{R}{r} & \calA_{n-2}(\vec{x}/n) \cdots \calA_{1}(\vec{x}/n) \calA_{0}(\vec{x}/n) \one_R^r \\
    & = C \Psi(1,\vec{x}) + O\bigl( \tfrac1m  (\tfrac{n}{m})^{2r(R-r)/\beta}\bigr) + O(\tfrac mn + \tfrac1m).
\end{align*}
Choosing $m$ comparable to $n^{1-\delta}$ with $\delta=\tfrac12\beta/[\beta+r(R-r)]$ we see that the error terms can be
made negligible in the $n\to\infty$ limit.  Thus the proposition follows.
\end{proof}

%
%
%
%

\section{Main Theorem and Applications}
\label{S:examples}

\begin{theorem}\label{T:main}
Given $\vec{w}\in\C^{R-r}$ and $\vec{y}\in\C^{r}$, let $\vec{x}$ be the vector formed by concatenating $w_j$ and $\bar
y_k$, as in \eqref{E:concat}, and let $\Psi$ be the (unique) solution of the initial value problem
\begin{equation}\label{E:ODE again}
\frac{d\ }{dt} \Psi(t,\vec{x}) = \bigl( \tfrac{2}{\beta t} \Delta + V \bigr) \Psi(t,\vec{x})
\qtq{and} t^{-2r(R-r)/\beta} \Psi(t) =  \one_R^r + O(t).
\end{equation}
Then
\begin{equation}\label{E:Answer}
\begin{aligned}
\lim_{n\to\infty} n^{-2r(R-r)/\beta} & \E^\beta_n\biggl\{ \prod_{j=1}^{R-r} Z_n(e^{iw_j/n}) \prod_{k=1}^r \overline{Z_n(e^{iy_k/n})} \biggr\} \\
   & = C \cdot \exp\Bigl\{\tfrac{i}2 \sum \bar y_k - \tfrac{i}2 \sum w_j \Bigr\}  \frac{\langle \one_R^r,\ \Psi(t=1,\vec{x} ) \rangle}{\|\chi_r^R\|^2}.
\end{aligned}
\end{equation}
Here $C$ is as in \eqref{E:Cdefn}.  Note that since $\|\chi_r^R\|^2 =\binom{R}{r}$, the limit takes the value $C$ when
$\vec{w}=0$ and $\vec{y}=0$.
\end{theorem}

\begin{proof}
Existence and uniqueness of solutions to \eqref{E:ODE again} was proved in Lemma~\ref{L:ODE}.  The formula \eqref{E:Answer} is
an immediate consequence of Proposition~\ref{P:v Answer}, \eqref{E:Z from v}, and \eqref{E:Z^R}.
\end{proof}

\begin{corollary}\label{C:Analytic}
LHS\eqref{E:Answer} is an analytic function of all arguments, more precisely, of each $w_j$, each $\bar y_k$, and of
$\beta$.
\end{corollary}

Analyticity in $\beta$ is the most interesting element of this corollary.  In statistical physics, failure of analyticity of quantities
in the thermodynamic limit is the signal of a phase transition.  Whether or not the (thermodynamic limit of the) point processes described
in \eqref{CGbeta} exhibit a phase transition is currently unresolved.  We believe that there is no phase transition and Corollary~\ref{C:Analytic}
supports this contention; however, (local) behaviour of the point process is more properly determined by the Laplace functionals (and
their analyticity), rather than autocorrelations of the characteristic polynomial.

\begin{corollary}\label{C:Answer 2}  Given $w,y\in\C$,
\begin{equation*}
\lim_{n\to\infty} n^{-2/\beta} \E^\beta_n\bigl\{ Z_n(e^{iw/n}) \overline{Z_n(e^{iy/n})} \bigr\}
   = \sqrt{\pi}\,e^{-i(w-\bar y)/2} (w-\bar y)^{\frac12 - \frac2\beta}
        J_{\frac2\beta-\frac12}\bigl(\tfrac{w-\bar y}2\bigr)
\end{equation*}
where $J_\nu$ denotes the Bessel function of the first kind and order $\nu$.
\end{corollary}

\begin{proof}
Choosing $(1,2),(2,1)$ as our ordering of the entries in $G_2^1$, the initial value problem \eqref{E:ODE again} reads
$$
\frac{d\ }{dt} \Psi(t) = \begin{bmatrix} \frac{i}2(\bar y -w) & \frac{2}{\beta t} \\ \frac{2}{\beta t} & \frac{i}2(w -\bar y) \end{bmatrix} \Psi(t)
    \qtq{and} \Psi(t) = t^{2/\beta} \begin{bmatrix} 1 + O(t) \\ 1+O(t) \end{bmatrix}.
$$
This has solution
$$
\Psi(t) = 2^{\frac4\beta-1} (w-\bar y)^{\frac12 - \frac2\beta} \Gamma(\tfrac2\beta+\tfrac12) t^{1/2}
    \begin{bmatrix} J_{\frac2\beta-\frac12}(\tfrac12[w-\bar y]t) + i J_{\frac2\beta+\frac12}(\tfrac12[w-\bar y]t) \\[1ex]
        J_{\frac2\beta-\frac12}(\tfrac12[w-\bar y]t) - i J_{\frac2\beta+\frac12}(\tfrac12[w-\bar y]t) \end{bmatrix}
$$
where $J_\nu$ denotes the Bessel function of the first kind.  Verifying that this is the desired solution requires only
a few of its basic properties:
\begin{align*}
J_\nu' = \tfrac{\nu}t J_\nu - J_{\nu+1} =  - \tfrac{\nu}t J_\nu + J_{\nu-1}
\quad\text{and}\quad
J_\nu (z) = \tfrac{1}{\Gamma(\nu+1)} \bigl( \tfrac{z}{2} \bigr)^\nu + O(z^{\nu+1}),
\end{align*}
Proofs of these can be found in \cite[\S\S4.5--6]{AndrewsAskeyRoy}.  Computing the value of $C$ via \eqref{E:Cdefn} and
then employing Legendre's duplication formula for the Gamma function (cf. \cite[Theorem~1.5.1]{AndrewsAskeyRoy}), we
obtain
$$
2^{\frac4\beta-1} \Gamma(\tfrac2\beta+\tfrac12) C = \sqrt{\pi}\,.
$$
The result now follows with a few more elementary manipulations.
\end{proof}

We have made some (computer assisted) investigations of higher values of $R$, resulting in rather long formulae involving
exotic functions (specifically, Whittaker and ${}_1F_2$ hypergeometric functions).  At present, it is not clear
to us that such results carry more or clearer information than the ODEs that generated them.

The moments of the characteristic polynomial at a single point on the circle have been known for a long time and can be
deduced from the Selberg integral (cf. \cite{ConreyFKRSnaith}):
\begin{align}
\E\Bigl\{ |Z_n(z)|^{2\lambda} \Bigr\} &=
    \prod_{l=0}^{n-1} \frac{ \Gamma(2\lambda+\tfrac\beta2l+1)\,\Gamma(\tfrac\beta2l +1) }%
    { \Gamma(\lambda+\tfrac\beta2l +1)^2 } \label{Gamma product}
\end{align}
for any $\lambda\in\C$ with $\Re \lambda \geq 0$.  This result can be easily obtained from the basic representation of $Z_n(z)$
used in this paper:  When $|z|=1$, \eqref{Phi*} shows that $z\Phi_{k}(z)/\Phi^*_{k}(z)$ is unimodular.  Next we note that by
\eqref{Z from Phi} and \eqref{SzegoRec},
\begin{align*}
|Z_n(z)|            = |\Phi^*_{n-1}(z)| \, \biggl| 1 - e^{i\eta} \frac{z\Phi_{n-1}(z)}{\Phi^*_{n-1}(z)} \biggr|
\ \text{ and }\ %
|\Phi^*_{k+1}(z)|   &= |\Phi^*_{k}(z)|  \, \biggl| 1 - \alpha_k  \frac{z\Phi_{k}  (z)}{\Phi^*_{k}  (z)} \biggr|,
\end{align*}
respectively. Combining this with
Remark~\ref{R:rot}, the statistical independence of the parameters, and Lemma~\ref{L:blah} shows
\begin{align*}
\E\bigl\{ |Z_n(z)|^{2\lambda}  \bigr\} &= \E\biggl\{ |1-e^{i\eta}|^{2\lambda} \; \prod_{k=0}^{n-2} |1-\alpha_k|^{2\lambda} \biggr\}
=\prod_{l=0}^{n-1} \frac{ \Gamma(2\lambda +\tfrac\beta2 l + 1)\,\Gamma(\tfrac\beta2 l +1) }%
    { \Gamma(\lambda+\tfrac\beta2 l +1)^2 }.
\end{align*}
Note that $l$ represents $k+1$, while $l=0$ captures the average over $\eta$.

Much of the argument just presented could already be seen in the arguments leading to Lemma~\ref{L:A prod asymp}.
An essentially equivalent proof of \eqref{Gamma product} can be found in \cite{BHNY}, albeit, restricted to the case $\beta=2$.
For completeness, we now state the special case of Theorem~\ref{T:main} that follows by setting $R=2r$ and all parameters equal:

\begin{corollary}\label{C:1 point}
For $x\in\R$ and non-negative $r\in\Z$,
\begin{align*}
\lim_{n\to\infty} n^{-2r^2/\beta} \E^\beta_n\bigl\{ |Z_n(e^{ix})|^{2r} \bigr\}
    = C
    = \prod_{p=1}^{r} \frac{\Gamma\bigl(\tfrac2{\beta}p\bigr)}{\Gamma\bigl(\tfrac2{\beta}(r+p)\bigr)} .
\end{align*}
\end{corollary}

%
%
%
%



\begin{thebibliography}{10}
\newcommand{\MSN}[1]{\href{http://www.ams.org/mathscinet-getitem?mr=#1}{\sc MR#1}}


\bibitem{AkemannV} G. Akemann and  G. Vernizzi,
Characteristic polynomials of complex random matrix models.
\emph{Nuclear Phys. B} \textbf{660} (2003), 532--556. \MSN{1982915}

\bibitem{AndreevSimons} A. V. Andreev and B. D. Simons,
Correlators of spectral determinants in quantum chaos.
\emph{Phys. Rev. Lett.} \textbf{75} (1995), 2304--2307.

\bibitem{AndrewsAskeyRoy} G. E. Andrews, R. Askey, and R. Roy,
\emph{Special functions.} Encyclopedia of Mathematics and its Applications, \textbf{71}. Cambridge University Press,
Cambridge, 1999. \MSN{1688958}

\bibitem{Aomoto} K.~Aomoto,
Jacobi polynomials associated with Selberg integrals.
\textit{SIAM J. Math. Anal.} {\bf 18} (1987), 545--549.

\bibitem{BaikDeiftStrahov} J. Baik, P. Deift, and E. Strahov,
Products and ratios of characteristic polynomials of random Hermitian matrices.
\emph{J. Math. Phys.} \textbf{44} (2003), 3657--3670. \MSN{2006773}

\bibitem{BasorForrester}
E. L. Basor and P. J. Forrester,
Formulas for the evaluation of Toeplitz determinants with rational generating functions.
\emph{Math. Nachr.} \textbf{170} (1994), 5--18. \MSN{1302362}

\bibitem{BrezinHikami} E.~Br\'ezin and S.~Hikami,
Characteristic polynomials of random matrices.
\emph{Comm. Math. Phys.} \textbf{214} (2000), 111--135. \MSN{1794268}

\bibitem{Berry} M.~V.~Berry,
Quantum Chaology, \textit{Proc. Roy. Soc. Lond. A} \textbf{413} (1987), 183--198.
\MSN{0909277}

\bibitem{Bohigas} O.~Bohigas, M.~Giannoni, and C.~Schmit,
Characterization of Chaotic Quantum Spectra and Universality of Level Fluctuation Laws.
\textit{Phys. Rev. Lett.} \textbf{52} (1984), 1--4.

\bibitem{BorodinStrahov} A.~Borodin and E.~Strahov,
Averages of characteristic polynomials in random matrix theory.
\emph{Comm. Pure Appl. Math.} \textbf{59} (2006), 161--253. \MSN{2190222}

\bibitem{BHNY}
P. Bourgade, C. P. Hughes, A. Nikeghbali, and M. Yor, The characteristic polynomial of a random unitary matrix: a
probabilistic approach. \emph{Duke Math. J.} \textbf{145} (2008), no. 1, 45--69. \MSN{2451289}

\bibitem{CMV} M.~J.~Cantero, L.~Moral, and L.~Vel\'azquez,
Five-diagonal matrices and zeros of orthogonal polynomials on the unit circle.
\textit{Linear Algebra Appl.} {\bf 362} (2003), 29--56. \MSN{1955452}

\bibitem{ConreyFKRSnaith} J. B. Conrey, D. W. Farmer, J. P. Keating, M. O. Rubinstein, and N. C. Snaith,
Autocorrelation of random matrix polynomials.
\emph{Comm. Math. Phys.} \textbf{237} (2003), no. 3, 365--395. \MSN{1993332}

\bibitem{DumE} I.~Dumitriu and A.~Edelman,
Matrix models for beta ensembles.
\textit{J. Math. Phys.} {\bf 43} (2002), 5830--5847.

\bibitem{Dyson} F.~Dyson,
Statistical theory of the energy levels of complex systems. I, II, and III.
\textit{J. Math. Phys.} {\bf 3} (1962), 140--156, 157--165, and 166--175. \MSN{0143556}, \MSN{0143557}, and \MSN{0143558}

\bibitem{Good} I.~J.~Good,
Short proof of a  conjecture by Dyson.
\textit{J. Math. Phys.} {\bf 11} (1970), 1884. \MSN{0258644}


\bibitem{Herz} C.~S.~Herz, Bessel functions of matrix argument.
Ann. of Math. \textbf{61}, (1955), 474--523. \MSN{0069960}



\bibitem{Kaneko} J. Kaneko,
Selberg integrals and hypergeometric functions associated with Jack polynomials.
\emph{SIAM J. Math. Anal.} \textbf{24} (1993), 1086--1110. \MSN{1226865}

\bibitem{KatzSarnak} N. M. Katz and P. Sarnak,
Zeroes of zeta functions and symmetry.
\textit{Bull. Amer. Math. Soc.} {\bf 36} (1999), 1--26.
\MSN{1640151}

\bibitem{KN1} R. Killip and I. Nenciu,
\textit{Matrix models for circular ensembles.}
Int. Math. Res. Not. \textbf{2004}, 2665-2701.




\bibitem{Mehta} M.~L.~Mehta,
\textit{Random matrices}, Second edition, Academic Press, Boston, MA, 1991.
\MSN{1083764}

\bibitem{MehtaN} M. Mehta and J. M. Normand,
Moments of the characteristic polynomial in the three ensembles of random matrices.
\emph{J. Phys. A} \textbf{34} (2001), 4627--4639. \MSN{1840299}

\bibitem{Montgomery} H. L. Montgomery,
The pair correlation of zeros of the zeta function.
In \emph{Analytic number theory}, Proc. Sympos. Pure Math., Vol. XXIV, pp. 181--193.
Amer. Math. Soc., Providence, R.I., 1973.
\MSN{0337821}


\bibitem{OPUC1} B.~Simon,
\textit{Orthogonal polynomials on the unit circle. Part 1. Classical theory.}
American Mathematical Society Colloquium Publications, 54, Part 1. American Mathematical Society, Providence, RI, 2005.
\MSN{2105088}

\bibitem{StrahovFyod} E. Strahov and Y. V. Fyodorov,
Universal results for correlations of characteristic polynomials: Riemann-Hilbert approach.
\emph{Comm. Math. Phys.} \textbf{241} (2003), 343--382. \MSN{2013803}


\bibitem{Trotter} H.~F.~Trotter,
Eigenvalue distributions of large Hermitian matrices; Wigner's semicircle law and a theorem of Kac, Murdock, and Szeg\H{o}.
\textit{Adv. in Math.} \textbf{54} (1984), 67--82.
\MSN{0761763}

\bibitem{Watkins} D.~Watkins,
Some perspectives on the eigenvalue problem. \textit{SIAM Rev.} \textbf{35} (1993), 430--471. \MSN{1234638}


\bibitem{Wilson} K.~Wilson,
Proof of a conjecture by Dyson.
\textit{J. Math. Phys.} {\bf 3} (1962), 1040--1043. \MSN{0144627}

\end{thebibliography}
\end{document}